\documentclass[10pt,a4paper]{article}
\usepackage{indentfirst}

\usepackage{graphicx}
\usepackage{amsfonts}
\usepackage{amstext}
\usepackage{amsopn}
\usepackage{color}
\usepackage{latexsym,bm}
\usepackage{amssymb}
\usepackage{amsmath}
\usepackage{placeins}
\usepackage{amsthm}
\usepackage{epstopdf}

\newtheorem{thm}{Theorem}[section]
\newtheorem{lem}[thm]{Lemma}

\newtheorem{rem}[thm]{Remark}

\begin{document}
\title{Capacity of the Adini element for biharmonic equations}
\author{\normalsize Jun Hu$^\dagger$,~~Xueqin Yang$^\star$,~~Shuo Zhang$^\ast$\\ \normalsize
$^\dagger$ LMAM and School of Mathematical Sciences,
Peking University, \\ \normalsize Beijing 100871, P.R.China\\
$^\star$ School of Mathematical Sciences,
Peking University, \\ \normalsize Beijing 100871, P.R.China\\
$^\ast$ LESC, ICMSEC, NCMIS, Academy of Mathematics and System Science,\\ Chinese Academy of Sciences,
\normalsize Beijing 100190, P.R.China\\
\\\vspace{2mm} \normalsize email: hujun@math.pku.edu.cn; ~~ yangxueqin1212@pku.edu.cn;\\ szhang@lsec.cc.ac.cn \normalsize
}
\date{}
\maketitle


\begin{abstract}
This paper is devoted to the convergence analysis of the Adini element scheme for the fourth order problem in any dimension.  We showed that, under the regularity assumption that the exact solution is in $H^4$, the Adini element scheme is $\mathcal{O}(h^2)$ order convergent in energy norm, and the convergence rate in $L^2$ norm can not be nontrivially higher than $\mathcal{O}(h^2)$ order. Numerical verifications are presented.
\end{abstract}

\tableofcontents

\section{Introduction}
\label{intro}
This paper is devoted to the convergence analysis of the Adini element scheme for the fourth order problem in any dimension. The Adini element, (c.f. \cite{AdiniClough} for 2D, \cite{WangShiXu} for higher dimension) is among the earliest finite elements for elliptic problems. It uses the rectangles (2D) and generalised rectangles (higher dimensions) as geometry shapes, and the evaluation and the derivatives of first order on the vertices as nodal parameters. The generation of stiffness matrix is easy and friendly, and this element has become a popular one during the past half century, and stimulated various works, (see \cite{LascauxLesaint}, \cite{MaoChen}, \cite{LuoLin}, \cite{WangShiXu}). In this present paper, we discuss the capacity of the Adini element scheme for fourth order problems, and present a sharp analysis of the upper and lower bound of the convergence rate in energy and integral norms in arbitrary dimensions.

When used for second order problems, the Adini element scheme is a conforming one, and the error analysis is straightforward by the fundamental C\'ea lemma and standard arguments. When used for fourth order problems, however, the Adini element is a nonconforming one, and the convergence analysis is more subtle. In Wang, Shi and Xu \cite{WangShiXu} where they generalised the Adini element from 2D to arbitrary dimension, the $\mathcal{O}(h)$ convergence rate of Adini element in any dimension has been proved for fourth order problem. Meanwhile, a higher accuracy of the scheme is still expected and numerically observed. In 2D, it has been proved by Lascaux and Lesaint \cite{LascauxLesaint} that the finite element solution converges to the exact solution with $\mathcal{O}(h^2)$ order in the energy norm, provided the rectangular cells in the grid are all the same. Then in 2004, Lin and Luo \cite{LuoLin} showed the $\mathcal{O}(h^2)$ convergence of the Adini element without assuming the congruence of the cells of the grid. Later in 2006, Mao and Chen \cite{MaoChen} showed further the $\mathcal{O}(h^2)$ convergence rate for anisotropic grids. So far to our knowledge, the sharp analysis of the convergence rate of Adini element for fourth order problems in higher dimensions is still absent.

In this paper, we study the convergence rate of the Adini element scheme for fourth order problems in higher dimensions. Technically, without making a crucial use of the nodal interpolation which was done by (\cite{LascauxLesaint}, \cite{LuoLin}, \cite{MaoChen}) and which will bring extra regularity assumption on the exact solution in higher dimension, our analysis relies on the structure of Adini element space only. We figure out the intrinsic symmetry property of the Adini element space, and show the $\mathcal{O}(h^2)$ energy norm convergence rate in a unified way with respect to the dimensions provided the exact solution belongs to $H^4$.

There have been works that study high accuracy nonconforming finite element methods for fourth order problems. Several $\mathcal{O}(h^2)$ nonconforming finite elements have been constructed in, e.g.,  \cite{ChenChenQiao,ShiChen,WangZuZhang}. In contrast to these elements, the Adini element space does not possess such moment continuities; the average of the normal derivatives of Adini element function is not continuous across the internal faces. This hints us to make use of a different way by using the symmetric property inside one cell, and moreover, this unusual property makes it hardly possible to make use of the dual argument to obtain higher order convergence rate in $H^1$ or $L^2$ norm. Indeed, in the paper, we further show that the convergence rate in $L^2$ norm can not be non-trivially higher than $\mathcal{O}(h^2)$ order.

The analysis of the lower bound of the convergence rate of the Adini element scheme in $L^2$ norm is a generalization of Hu-Shi's work \cite{HuShi}, which solved an open problem whether the convergent order in $L^2$ norm can always be higher than that in the energy norm. Technically, a decomposition of the residual $(f,u-u_h)$ to a leading term and other higher order terms works crucially, and we estimate the lower bound of the leading term sufficiently. Again, a sharp analysis of the interpolation operator will play a key role. Therefore, by the discrete Poincar\'e inequality, we obtain that the convergence rate of the Adini element scheme for fourth order problem in energy norm, $H^1$ norm and the integral norm are all of $\mathcal{O}(h^2)$ order, and these estimates are all sharp.

The remaining of the paper is organized as follows. In Section \ref{sec:preliminaries}, we present some preliminaries of the Adini element. In Section \ref{sec:model problem}, we present the model problem and the Adini finite element discretization. In Section \ref{sec:energynorm}, we show the $\mathcal{O}(h^2)$ order convergence rate in energy norm in any dimension. In Section \ref{sec:l2norm},  we further show the $\mathcal{O}(h^2)$ order convergence in $L^2$ norm in any dimension. In Section \ref{sec:numerical}, some numerical examples are presented to demonstrate our theoretical results. Finally, in Section \ref{sec:conclusion}, some conclusions are given.

\section{The preliminaries: the Adini element}
\label{sec:preliminaries}

\subsection{The Adini element}

Let $K\subset \mathbb{R}^d$ be a $d$-rectangle, $x_c=(x_{1,c},x_{2,c},\cdots,x_{d,c})^\mathrm{T}\in \mathbb{R}^d$ be the barycentre of $K$, and $h_i$ the half length of $K$ in $x_i$ direction, $i=1,2,\dots,d$. Then the $d$-rectangle can be denoted by
\begin{equation*}
K=\{x=(x_1,x_2,\cdots,x_d)^\mathrm{T}\,|\,x_i=x_{i,c}+\xi_ih_i,\;-1\leq \xi_i\leq 1,\;1\leq i\leq d\}.
\end{equation*}
Particularly, the vertices $a_i,\,1\leq i\leq 2^d,$ of $K$ are denoted by
\begin{equation*}
a_i=(x_{1,c}+\xi_{i1}h_1,x_{2,c}+\xi_{i2}h_2,\cdots,x_{d,c}+\xi_{id}h_d)^\mathrm{T},\; |\xi_{ij}|=1,\ 1\leq j\leq d,\ 1\leq i\leq 2^d.
\end{equation*}
Moreover, denote by $F'_{K,i}$ and $F''_{K,i}$ the two $(d-1)$-dimensional faces of $K$ without the edges parallel to the $x_i$ axe; see Figure 1.

The d-dimensional Adini element is defined by the triple $(K,P_A(K),D)$, where
\begin{itemize}
\item the geometric shape $K$ is a $d$-rectangle;
\item the shape function space is
\begin{equation}
P_A(K):=Q_1(K)+\text{span}\{x_i^2q\;|\,1\leq i\leq d,\; q\in Q_1(K)\},
\end{equation}
here and throughout this paper, $Q_l(K)$ denotes the space of all polynomials which are of degree$\leq l$ with respect to each variable $x_i$, over $K$;
\item the nodal parameters are, for any $v\in C^1(K)$,
\begin{equation}
D(v):=\bigg(v(a_i),\;\;\nabla v(a_i)\bigg),\;\;
\end{equation}
where $a_i$ are vertices of $K$, $i=1,\dots,2^d$.
\end{itemize}

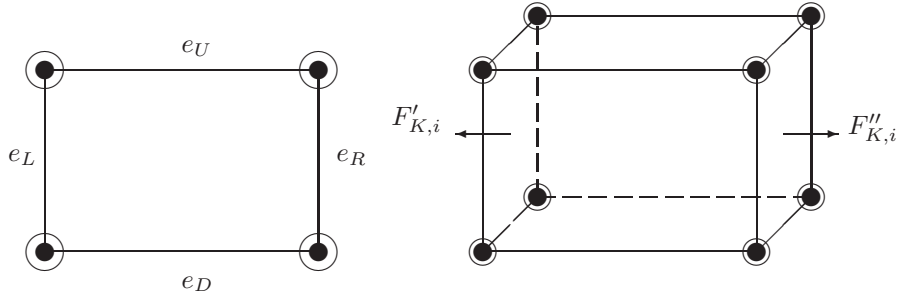
\begin{figure}
%\label{fig1}
\begin{center}
\setlength{\unitlength}{2.4cm}
\begin{picture}(2,2)
\put(-1.,0.5){\line(1,0){1.5}} \put(-1,0.5){\line(0,1){1}}
\put(0.5,0.5){\line(0,1){1}} \put(-1,1.5){\line(1,0){1.5}}
\put(-1.,0.5){\circle*{0.1}}
\put(-1.,0.5){\circle{0.2}}
\put(-1.,1.5){\circle*{0.1}}
\put(-1.,1.5){\circle{0.2}}
\put(0.5,0.5){\circle*{0.1}}
\put(0.5,0.5){\circle{0.2}}
\put(0.5,1.5){\circle*{0.1}}
\put(0.5,1.5){\circle{0.2}}

%\put(1,0.5){\vector(0,-1){0.3}}
\put(-0.25,0.3){$ e_D$}
\put(-0.25,1.6){$ e_U$}
\put(-1.2,1){$ e_L$}
\put(0.6,1){$ e_R$}
%\put(1,1.5){\vector(0,1){0.3}}
%\put(0,1){\vector(-1,0){0.3}}
%\put(2,1){\vector(1,0){0.3}}

%\label{fig2}

\put(1.4,0.5){\line(1,0){1.5}} \put(1.4,0.5){\line(0,1){1}}
\put(2.9,0.5){\line(0,1){1}} \put(1.4,1.5){\line(1,0){1.5}}
\put(1.4,0.5){\circle*{0.1}}
\put(1.4,0.5){\circle{0.15}}
\put(1.4,1.5){\circle*{0.1}}
\put(2.9,0.5){\circle*{0.1}}
\put(2.9,1.5){\circle*{0.1}}
\put(1.4,1.5){\circle{0.15}}
\put(2.9,0.5){\circle{0.15}}
\put(2.9,1.5){\circle{0.15}}
\put(1.4,1.5){\line(1,1){0.3}}
\put(2.9,1.5){\line(1,1){0.3}}
\put(1.7,1.8){\line(1,0){1.5}}
\put(1.7,1.8){\circle*{0.1}}
\put(3.2,1.8){\circle*{0.1}}
\put(1.7,1.8){\circle{0.15}}
\put(3.2,1.8){\circle{0.15}}
\put(2.9,0.5){\line(1,1){0.3}}
\put(3.2,0.8){\circle*{0.1}}
\put(3.2,0.8){\circle{0.15}}
\put(3.2,0.8){\line(0,1){1}}
\multiput(1.4,0.5)(0.17,0.17){2}{\line(1,1){0.15}}
%\put(2.5,0.5){\line(1,1){0.3}}
\multiput(1.7,0.8)(0.14,0){11}{\line(1,0){0.1}}
\multiput(1.7,0.8)(0,0.14){7}{\line(0,1){0.1}}
\put(1.7,0.8){\circle*{0.1}}
\put(1.7,0.8){\circle{0.15}}
%\put(2.2,0.65){\vector(0,-1){0.3}}
%\put(2.2,1.65){\vector(0,1){0.3}}
\put(1.55,1.15){\vector(-1,0){0.3}}
\put(3.05,1.15){\vector(1,0){0.3}}
%\put(2,1){\vector(-1,-1){0.25}}
%\put(2.3,1.3){\vector(1,1){0.25}}
%\put(2.2,0.25){$\iint F_6$}
%\put(2.2,1.95){$\iint F_3$}
\put(0.9,1.2){$ F'_{K,i}$}
\put(3.4,1.15){$ F''_{K,i}$}

\linethickness{0.6mm}
\end{picture}
\end{center}
\caption{ degrees of freedom for the Adini element}
\end{figure}

Let $\alpha$ denote the multiple-index with $\alpha=(\alpha_1,\cdots,\alpha_d)$, $\alpha_i(1\leq i\leq d)$ are nonnegative integers, and $|\alpha|=\sum\limits_{i=1}^d\alpha_i$,  $x^{\alpha}=\prod\limits_{i=1}^d x_i^{\alpha_i}$. The partial derivative operator can be written as
\begin{equation*}
\partial^{\alpha}=\frac{\partial^{|\alpha|}}{\partial x_1^{\alpha_1}\cdots x_d^{\alpha_d}}.
\end{equation*}
Let $e_i$ $( 1\leq i\leq d)$ be the $d-$dimensional unit multi-index with its $i-$th entry equal to 1.

%{\color{red} $\alpha$ d-index, $\xi^\alpha, \;\xi^{\alpha}=\prod_{k=1}^d \xi_k^{\alpha_k}$, reference cell.}

%
%
\subsection{Structural properties of the Adini element}

Given $\Omega$ a  $d$-dimensional domain, $\mathcal{T}_h$ is a triangulation on $\Omega$, and $K\in \mathcal{T}_h$. Let $\Pi^1_K$ be piecewise bilinear interpolation operator on $K$, namely $\Pi^1_K v\in Q_1(K)$ and $(\Pi^1_K v)(P)=v(P),\;\text{for any vertex}\; P\;of\;K$, and $v\in C(K)$.  Define on $C(K)$ the operator $\mathcal{R}^1_K:=Id-\Pi_K^1$, with $Id$ being the identity operator. Define $\Pi_{0,K}w=\frac{1}{|K|}\int_Kw\mathrm{d}x$, for any $w\in L^2(K)$. The global version $\Pi_0$ of the interpolation operator $\Pi_{0,K}$ is defined as
\begin{equation} \Pi_0|_K=\Pi_{0,K},\;\text{for any}\;K \in \mathcal{T}_h. \end{equation}

\begin{lem}\label{lem8}
It holds for $w_h\in P_A(K)$ that
\begin{equation}
\mathcal{R}^1_K\frac{\partial w_h}{\partial x_i}\bigg|_{F'_{K,i}}=\mathcal{R}^1_K\frac{\partial w_h}{\partial x_i}\bigg|_{F''_{K,i}},\;1\leq i\leq d.\label{eq8}
\end{equation}
\end{lem}
\begin{proof}
Given $w_h \in P_A(K)$, a direct calculation leads to that
%It follows from the definitions of $P_A(\widehat{K})$ and $\mathcal{R}^1_{\widehat{K}}$ that, for any $\widehat{w}_h\in P_A(\widehat{K})$,

\begin{equation}
\frac{\partial w_h}{\partial x_i}\in Q_1(K)+{\rm span}\left\{(x_j-x_{j,c})^2\cdot\widehat{q},\ \widehat{q}\in Q_1^i(K),1\leq j\leq d\right\},
\end{equation}
where $Q_1^i(K):={\rm span}\{(x-x_c)^\alpha\}_{|a_j|\leq 1,\alpha_i=0}$. Further,
\begin{equation}\label{eq40}
 \mathcal{R}^1_K\big(\frac{\partial w_h}{\partial x_i}\big)\in S^i_K,\; S^i_K={\rm span}\{((x_j-x_{j,c})^2-h^2_j)\cdot \widehat{q}, 1\leq j\leq d, \widehat{q}\in Q^i_1(K) \}.
\end{equation}
Noting that $(x_j-x_{j,c})^2$, $1\leq j\leq d$, evaluate the same on $F'_{K,i}$ and $F''_{K,i}$, we obtain \eqref{eq8}. This finishes the proof.

\end{proof}
For ease of expression, we define the following sets
\begin{equation*}
\begin{split}
M_{i,j}&=\{(\alpha_1,\cdots,\alpha_d)|\;\alpha_i=1,\;2\leq \alpha_j\leq 3,\;\alpha_k\leq 1,\;k\neq i,j\},
\\
M'_{1,j}&=\{(\alpha_1,\alpha_2,\alpha_3)|\;\alpha_1=1,\;2\leq \alpha_j\leq 3,\;\alpha_k\leq 1,\;2\leq k\neq j\leq 3\}.
\end{split}
\end{equation*}
By means of (\ref{eq40}), on $F'_{K,i}$, $F''_{K,i}$ of the element $K$, we can get that
\begin{equation}
\label{eq:expansionRKface}
\begin{split}
&\left(\mathcal{R}^1_{K}\frac{\partial w_h}{\partial x_i}\right)(x_1,\cdots,x_{i-1},x_{i,c}\pm h_i,x_{i+1},\cdots,x_d)\\=&\sum_{\substack{1\leq j\leq d\\ j\neq i}}\sum_{\alpha\in M_{i,j}}B{^K_i}(j,\alpha)\Pi_{0,K}(\partial^{\alpha}w_h),
 \end{split}
 \end{equation}
with
%\begin{equation}
%B^K_i(j,\alpha) = \left\{
%\begin{array}{l}
%\frac{1}{2}\big[(x_j-x_{j,c})^2-h^2_j\big](x-x_c)^{\alpha-e_i-2e_j},\;
%\vspace{2mm}
%\\
%\qquad \qquad \qquad \alpha_i =1,\;\alpha_j =2,\;|\alpha_k|\leq 1,\;k\neq i\neq j,
%\vspace{2mm}
%\\
%\frac{1}{6}\big[(x_j-x_{j,c})^3-h^2_j(x_j-x_{j,c})\big](x-x_c)^{\alpha-e_i-3e_j},\;
%\vspace{2mm}
%\\
%\qquad \qquad \qquad \alpha_i =1,\;\alpha_j =3,\;|\alpha_k|\leq 1,\;k\neq i\neq j,
%\vspace{2mm}
%\\
%0,\qquad\qquad \quad otherwise,
%\end{array}
%\right.
%\end{equation}
\begin{equation}
B^K_i(j,\alpha)=\frac{1}{\alpha_j!}\big[(x_j-x_{j,c})^{\alpha_j}-h_j^2(x_j-x_{j,c})^{\alpha_j-2}\big](x-x_c)^{\alpha-e_i-\alpha_je_j},
\end{equation}
\begin{equation}\label{eq54}
\partial^{\alpha}w_h\in\text{span}\{1,x_j\},\; \text{if}\;\alpha_j=2,\;\text{and}\; \partial^{\alpha}w_h\; \text{is constant, if}\;\alpha_j=3.
\end{equation}
Noticing that $\partial_{x_i}B^K_i(j,\alpha)=0$, if $i\neq j$, and $\Pi_{0,K}(\partial^{\alpha}w_h)$ are constant.\\
%\begin{equation}
%\overline{\partial^{\alpha}w_h} = \left\{
%\begin{array}{l}
%\frac{1}{|K|}\int_K\partial^{\alpha}w_h\;\mathrm{d}x,
%\vspace{2mm}
%\\
%\qquad \qquad \qquad \alpha_i =1,\;\alpha_j =2,\;|\alpha_k|\leq 1,\;k\neq i\neq j,
%\vspace{2mm}
%\\
%\partial^{\alpha}w_h,\quad \qquad \alpha_i =1,\;\alpha_j =3,\;|\alpha_k|\leq 1,\;k\neq i\neq j,
%\vspace{2mm}
%\\
%0,\qquad\qquad \quad otherwise.
%\end{array}
%\right.
%\end{equation}

%\begin{equation}
%\overline{\partial^{\alpha}w_h}=\frac{1}{|K|}\int_K\partial^{\alpha}w_h\;\mathrm{d}x,\;\alpha\in M_{i,j}.
%\end{equation}
For example, in two-dimensional case,
\begin{equation*}
\begin{split}
\mathcal{R}^1_{K}\frac{\partial w_h}{\partial x_1}(x_{1,c}\pm h_1,x_2)=&\frac{1}{2}[(x_2-x_{2,c})^2-h_2^2]\Pi_{0,K}\frac{\partial^3w_h}{\partial x_1\partial x_2^2}
\\
&+\frac{1}{6}[(x_2-x_{2,c})^3-h_2^2(x_2-x_{2,c})]\frac{\partial^4w_h}{\partial x_1\partial x_2^3},
\end{split}
\end{equation*}
and in three-dimensional case
\begin{equation*}
\begin{split}
&\mathcal{R}^1_{K}\frac{\partial w_h}{\partial x_1}(x_{1,c}\pm h_1,x_2,x_3)
\\
=&\sum_{j=2}^3\sum_{\alpha'\in M'_{1,j}}\frac{1}{\alpha_j!}[(x_j-x_{j,c})^{\alpha_j}-h_j^2(x_j-x_{j,c})^{\alpha_j-2}](x-x_c)^{\alpha-e_1-\alpha_je_j}\Pi_{0,K}(\partial^{\alpha'}w_h)
%\\
%=&\frac{1}{2}[(x_2-x_{2,c})^2-h_2^2]\Pi_{0,K}\frac{\partial^3w_h}{\partial x_1\partial x_2^2}+\frac{1}{2}[(x_3-x_{3,c})^2-h_3^2]\Pi_{0,K}\frac{\partial^3w_h}{\partial x_1\partial x_3^2}
%\\
%&+\frac{1}{2}[(x_2-x_{2,c})^2-h_2^2](x_3-x_{3,c})\Pi_{0,K}\frac{\partial^4w_h}{\partial x_1\partial x_2^2\partial x_3}
%\\
%&+\frac{1}{2}[(x_3-x_{3,c})^2-h_3^2](x_2-x_{2,c})\Pi_{0,K}\frac{\partial^4w_h}{\partial x_1\partial x_2\partial x_3^2}
%\\
%&+\frac{1}{6}[(x_2-x_{2,c})^3-h_2^2(x_2-x_{2,c})\Pi_{0,K}\frac{\partial^4w_h}{\partial x_1\partial x_2^3}
%\\
%&+\frac{1}{6}[(x_3-x_{3,c})^3-h_3^2(x_3-x_{3,c})]\Pi_{0,K}\frac{\partial^4w_h}{\partial x_1\partial x_3^3}
%\\
%&+\frac{1}{6}[(x_2-x_{2,c})^3-h_2^2(x_2-x_{2,c})](x_3-x_{3,c})\Pi_{0,K}\frac{\partial^5w_h}{\partial x_1\partial x_2^3\partial x_3}
%\\
%&+\frac{1}{6}[(x_3-x_{3,c})^3-h_3^2(x_3-x_{3,c})](x_2-x_{2,c})\Pi_{0,K}\frac{\partial^5w_h}{\partial x_1\partial x_2\partial x_3^3}.
\end{split}
\end{equation*}
%\begin{equation*}
%\mathcal{R}^1_{K}\frac{\partial w_h}{\partial x_1}(x_{1,c}\pm h_1,x_2,x_3)=\sum_{j=2}^3\sum_{\alpha\in M_{i,j}}\frac{1}{\alpha_j!}\big[(x_j-x_{j,c})^{\alpha_j}-h_j^2(x_j-x_{j,c})^{\alpha_j-2}\big](x-x_c)^{\alpha-e_i-\alpha_je_j},\;\alpha\in M_{i,j}.
%\end{equation*}
Given $K\in \mathcal{T}_h$, we define the canonical interpolation operator $\Pi_K : C^1(K)\rightarrow P_A(K)$ by, for any $v \in C^1(K) $,\begin{equation}
 (\Pi_K v)(P)=v(P)\; \text{and}\; (\nabla\Pi_K v)(P)=\nabla v(P),
 \end{equation}\\
for any vertex $P$ of $K$. The interpolation operator $\Pi_K$ has the following error estimates:
\begin{equation}
|v-\Pi_K v|_{l,K}\leq Ch^{4-l}|v|_{4,K},\;l=0,1,2,3,4,
 \end{equation}\\
provided that $v \in H^s(K)$, where $s\geq 4$ and $s>\frac{d}{2}+1$ such that $H^s(K)\subset C^1(K)$, see Remark \ref{rem1}.

\begin{lem}\label{lem4}
  For any $u\in P_4(K)$ and $v\in P_A(K)$, it holds that
  \begin{equation}
(\nabla^2(u-\Pi_Ku),\nabla^2v)_{L^2(K)} = -{\sum_{i=1}^d\sum_{\substack{1\leq j\leq d\\j\neq i}}}\frac{h_j^2}{3}\int_K\frac{\partial^4u}{\partial x_i^2\partial x_j^2}\frac{\partial^2v}{\partial x_i^2}\,\mathrm{d}x,
\end{equation}
%{\color{red} do you mean $\sum_{i=1}^d\sum_{j\neq i}\frac{h_j^2}{3}\int_K ... $????}
%where we denote $\mathrm{d}X=\mathrm{d}x_1\cdots\mathrm{d}x_d$. $\mathrm{d}x$
  \end{lem}
  \begin{proof}
%  Consider the reference element $\widehat{K}$, the affine mapping
%  \begin{equation}
%  \xi_i=\frac{x_i-x_{i,c}}{h},
%  \end{equation}
%  where $(...,x_{i,c},...)$ are center of $K$.\\
  %$h_i$ are length of $K$.\\
  It follows from the definition of $P_A(K)$ that
%{\color{red} in the equation below, there is $i$ in the left hand side, but no $i$ in the right hand side. is it really right?}
%{\color{green}Yes, because $\frac{\partial^2 v}{\partial x_i^2}\in Q_1(K)$. Or separate the constant term from $Q_1(K)$ } {\color{red} I am not sure if you are right. At least, you may need an index $i$ for $c_k$.}
\begin{equation}\label{eq20}
  \begin{split}
  \frac{\partial^2 v}{\partial x_i^2}&\in Q_1(K),
  \\
  \frac{\partial^2 v}{\partial x_i\partial x_j}&\in \{p\cdot \tilde{q}\,|\,p\in P_1(x_i,x_j),\;\tilde{q}\in Q_1^{i,j}(K)\}
  \\
  &\;\;\;\;+\text{span}\{x_k^2\cdot \tilde{q},\;\tilde{q}\in Q_1^{i,j}(K),\;1\leq k\leq d\},\;i\neq j,
  \end{split}
  \end{equation}
  where $P_1(x_i,x_j):=\text{span}\{1,\;x_i,\;x_j\}$, $Q_1^{i,j}(K):=\text{span}\{x^{\alpha}\}_{|\alpha_k|\leq 1,\alpha_i=\alpha_j=0}$.
  %, $\overline{f}=\frac{1}{|K|}\int_Kf\mathrm{d}x$.

Since $u\in P_4(K)$, we have, with $\xi_i=\frac{x_i-x_{i,c}}{h_i}$,
  \begin{equation}\label{eq21}
  u=u_1+\frac{h_i^4}{4!}\sum_{i=1}^d \frac{\partial^4 u}{\partial x_i^4}\xi_i^4+\frac{h_i^2h_j^2}{4}\sum_{i=1}^d\sum_{\substack{1\leq j\leq d\\j>i}}\frac{\partial^4u}{\partial x_i^2x_j^2}\xi_i^2\xi_j^2,
  \end{equation}
  where $u_1\in P_A(K)$.\\
  The Taylor expansion and the definition of the operator $\Pi_K$ yield
  \begin{equation}\label{eq22}
  u-\Pi_Ku=\frac{h_i^4}{4!}\sum_{i=1}^d \frac{\partial^4 u}{\partial x_i^4}(\xi_i^2-1)^2+\frac{h_i^2h_j^2}{4}\sum_{i=1}^d\sum_{\substack{1\leq j\leq d\\j>i}}\frac{\partial^4u}{\partial x_i^2x_j^2}(\xi_i^2-1)(\xi_j^2-1).
  \end{equation}
  Thus
  \begin{eqnarray*}
  \frac{\partial^2(u-\Pi_Ku)}{\partial x_i^2}&=&\frac{h_i^2}{4!}\frac{\partial^4 u}{\partial x_i^4}(12\xi_i^2-4)+\frac{h_j^2}{2}\frac{\partial^4u}{\partial x_i^2x_j^2}(\xi_j^2-1),\;1\leq i< j\leq d,\\
  \frac{\partial^2(u-\Pi_Ku)}{\partial x_i\partial x_j}&=&h_ih_j\frac{\partial^4u}{\partial x_i^2x_j^2}\xi_i\xi_j,\;1\leq i< j\leq d.
  \end{eqnarray*}
  Since
  \begin{equation}
  \int_K(12\xi_i^2-4)q_k\mathrm{d}x=0,\;q_k\in Q_1(K),\;\;1\leq k\neq i\leq d,
  \end{equation}
  and
 \begin{equation}
 \int_K(\xi_j^2-1)\xi_i\mathrm{d}x=0,\; \int_K(\xi_j^2-1)\xi_j\mathrm{d}x=0,\; \int_K(\xi_j^2-1)\xi_i\xi_j\mathrm{d}x=0,\;1\leq i\neq j\leq d,
 \end{equation}
 a combination of (\ref{eq20}) and (\ref{eq22}) and some elementary calculation yield
 \begin{equation}
 \int_K\frac{\partial^2(u-\Pi_Ku)}{\partial x_i^2} \frac{\partial^2 v}{\partial x_i^2}\;\mathrm{d}x=-\frac{h_j^2}{3}\int_K\frac{\partial^4u}{\partial x_i^2\partial x_j^2}\frac{\partial^2v}{\partial x_i^2}\,\mathrm{d}x,\;1\leq i\neq j\leq d.
 \end{equation}
 By the same argument, it yields
 \begin{equation}
 \int_K \frac{\partial^2(u-\Pi_Ku)}{\partial x_i\partial x_j} \frac{\partial^2 v}{\partial x_i\partial x_j}\mathrm{d}x=0,
 \end{equation}
which completes the proof.
% \qed
\end{proof}

\section{The capacity of Adini element for fourth order problems}

\subsection{Model problem and finite element discretisation}\label{sec:model problem}

Let $\Omega\subset\mathbb{R}^d$ be a bounded domain with Lipschitz boundary. In this paper, We consider the model fourth order elliptic problem:
\begin{eqnarray}
&&\left\{
\begin{array}{lll}
\Delta^2 u=f,\quad \mbox{in} \;\Omega,\\
u=\frac{\partial u}{\partial n}=0,\quad \mbox{on}\; \partial\Omega. \label{eq3}
\end{array}
\right.\\\nonumber
\end{eqnarray}
The variational formulation is, given $f\in H^{-2}(\Omega)$, to find $u\in V:=H^2_0(\Omega)$, such that
\begin{equation}
 a_\Omega(u,v)=(f,v),\;\text{for  any}\;\; v \in V. \label{eq7}
\end{equation}
where $\displaystyle a_{\Omega}(u,v):=\sum_{i,j=1}^d\int_{\Omega}\partial_{ij}u\partial_{ij}v$ for $u,v\in H^2(\Omega)$.

Let $\mathcal{T}_h$ be a regular $d$-rectangle triangulation of the domain $\Omega$. Define the Adini element space in a standard way by
\begin{eqnarray*}
V_h&:=&\{v \in L^2(\Omega):\;v|_K \in P_A(K),\; \forall K \in \mathcal{T}_h, \;v\;\mbox{and}\;\nabla v\;\mbox{is\;continuous\;at}\\&&\mbox{\;all\;internal\;vertices}\},
\end{eqnarray*}
and associated with the boundary condition,
\begin{eqnarray*}
V_{h0}&:=&\{v_h \in V_{h}:\ v_h\ \mbox{and}\ \nabla v_h\ \mbox{vanishes\;at\;all\;boundary\;vertices}\}.
\end{eqnarray*}
Evidently, $V_h\subset H^1(\Omega)$ and $V_{h0}\subset H^1_0(\Omega)$ (\cite{ShiWang}). However, $V_h\not\subset H^2(\Omega)$, and $V_{h0}\not\subset H^2_0(\Omega)$. Evidently, $P_3(K)\subset P_A(K)$ for any $K\in \mathcal{T}_h$. By the standard technique,
\begin{equation}  \label{eq18}
\inf_{v_h\in V_{h0}}|v- v_h|_{l,h}\leq Ch^{4-l}|v|_{4,\Omega},\;l=0,1,2,3,4,
 \end{equation}
for any $v \in H^4(\Omega)$. Herein and throughout this paper, $C$ denotes a generic positive constant which is independent of the meshsize and may be different at different places.

Associated with the model problem, the Adini finite element problem is to find $u_h \in V_{h0}$, such that
\begin{equation}
a_h(u_h,v_h)=(f,v_h)_{L^2(\Omega)},\; \text{for any}\; v_h \in V_{h0}, \label{eq5}
\end{equation}
where $\displaystyle a_h(u_h,v_h):= \sum_{K\in\mathcal{T}_h}a_K(u_h,v_h)$.

Define a semi-norm over $V_h$ by
%\begin{equation}
$|u_h|_h^2:=\sum_{K\in\mathcal{T}_h}\|\nabla^2u_h\|_{0,K}^2$. By Poincare inequality, $|\cdot|_h$ is a norm on $V_{h0}$, and it is equivalent to $\|\cdot\|_{h}$, while the latter denotes the piecewise $H^2$ norm.
%\end{equation}

%
%
\subsection{Error analysis in energy norm}
\label{sec:energynorm}

In this section, we present an upper bound of the energy norm of the error of the finite element scheme \eqref{eq5}. A main result of this paper is the theorem below.
\begin{thm}\label{thm:enesti}
Let $u$ and $u_h$ be the solutions of $(\ref{eq3})$ and \eqref{eq5}, respectively. Assume that $u\in H^4(\Omega)$. Then
\begin{equation}\label{eq:enesti}
  |u-u_h|_{2,h}\leq Ch^2|u|_{4,\Omega}.
\end{equation}
\end{thm}
\begin{proof}
By the second Strang Lemma, we have
%\begin{equation}
%|u-u_h|_h\leq C\bigg(\inf\limits_{v_h\in V_h}|u-v_h|_h+\sup\limits_{0\neq w_h\in V_h}\frac{|a_h(u,w_h)-(f,w_h)|}{|w_h|_h}\bigg),
%\end{equation}
\begin{equation}
|u-u_h|_h\leq C\big(\inf_{v\in V_{h0}}|u-v|_h+\sup_{w_h\in V_{h0}}\frac{|E_h(u,w_h)|}{|w_h|_h} \big), \label{eq6}
\end{equation}
where
\begin{eqnarray}
E_h(u,w_h)&:=&a_h(u,w_h)-(f,w_h) = \sum_{K\in \mathcal{T}_h}\int_{\partial K}\frac{\partial^2 u}{\partial n^2}\frac{\partial w_h}{\partial n}\mathrm{d}s.
\end{eqnarray}
The first term of (\ref{eq6}) is the approximation error and the second one is the consistency error.

%$$
%E_h(u,w_h)=\sum_{1\leq i\leq d}\sum_{K\in\mathcal{T}_h}\int_{\partial K}\frac{\partial^2 u}{\partial n^2}\frac{\partial w_h}{\partial x_i}n_ids
%$$

We shall consider separately the faces orthogonal to the $x_i$ axes $(1\leq i\leq d)$, namely we rewrite the consistency error to
\begin{equation}
E_h(u,w_h)=\sum_{i=1}^d E_{x_i}(u,w_h),
\end{equation}
with
\begin{eqnarray}
E_{x_i}(u,w_h)&=&\sum_{K\in \mathcal{T}_h}\int_{\partial K}\frac{\partial^2 u}{\partial n^2}\frac{\partial w_h}{\partial x_i}n_{x_i}\mathrm{d}s\nonumber
\\
&=&\sum_{K\in\mathcal{T}_h}\int_{\partial K}\frac{\partial^2 u}{\partial n^2}\mathcal{R}_K^1\frac{\partial w_h}{\partial x_i}n_{x_i}\mathrm{d}s\nonumber
\\
&=&  \sum_{K\in\mathcal{T}_h}(\int_{F_{K,i}''}-\int_{F_{K,i}'}) \frac{\partial^2u}{\partial x_i^2}\mathcal{R}_K^1\frac{\partial w_h}{\partial x_i}ds\nonumber
\\
&:=& \sum_{K\in\mathcal{T}_h}I_i^K(\frac{\partial^2 u}{\partial x_i^2},\mathcal{R}^1_K\frac{\partial w_h}{\partial x_i}),
\end{eqnarray}
where $n_{x_i}$ is the unit outward normal parallel to the $x_i$ axe.\\
Let $K\in\mathcal{T}_h$ and denote $g=\frac{\partial^2u}{\partial x_i^2}|_K$. It holds that, by \eqref{eq:expansionRKface},
 \begin{equation}\label{eq50}
 \begin{split}
 I_i^K(g,\mathcal{R}^1_K\frac{\partial w_h}{\partial x_i})&=\bigg(\int_{F_{K,i}''}-\int_{F_{K,i}'}\bigg)g\mathcal{R}^1_K\frac{\partial w_h}{\partial x_i}\mathrm{d}s
\\
&=\bigg(\int_{F''_{K,i}}-\int_{F'_{K,i}}\bigg)g\sum_{\substack{1\leq j\leq d\\ j\neq i}}\sum_{\alpha\in M_{i,j}}B^K_i(j,\alpha)\Pi_{0,K}\partial^{\alpha}w_h\mathrm{d}s
\\
&=\int_K\frac{\partial g}{\partial x_i}\sum_{\substack{1\leq j\leq d\\ j\neq i}}\sum_{\alpha\in M_{i,j}}B^K_i(j,\alpha)\Pi_{0,K}\partial^{\alpha}w_h\mathrm{d}x
\\
&=\int_K\frac{\partial g}{\partial x_i}\sum_{\substack{1\leq j\leq d\\ j\neq i}}\sum_{\alpha\in M_{i,j}}B^K_i(j,\alpha)\partial^{\alpha}w_h\mathrm{d}x
\\
&\;\;+\int_K\frac{\partial g}{\partial x_i}\sum_{\substack{1\leq j\leq d\\ j\neq i}}\sum_{\alpha\in M_{i,j}}B^K_i(j,\alpha)(\Pi_{0,K}-Id)\partial^{\alpha}w_h\mathrm{d}x
\\&:=L^K_{i,1}+L^K_{i,2}.
\end{split}
\end{equation}
Integrating by parts yields
\begin{equation}
\begin{split}
L^K_{i,1}=&\int_K\frac{\partial g}{\partial x_i}\sum_{\substack{1\leq j\leq d\\ j\neq i}}\sum_{\alpha\in M_{i,j}}B^K_i(j,\alpha)\partial^{\alpha}w_h\mathrm{d}x
\\
=&-\int_K\frac{\partial^2g}{\partial x_i^2}\sum_{\substack{1\leq j\leq d\\ j\neq i}}\sum_{\alpha\in M_{i,j}}B^K_i(j,\alpha)\partial^{\alpha-e_i}w_h\mathrm{d}x
\\
&+\bigg(\int_{F''_{K,i}}-\int_{F'_{K,i}}\bigg)\frac{\partial g}{\partial x_i}\sum_{\substack{1\leq j\leq d\\ j\neq i}}\sum_{\alpha\in M_{i,j}}B^K_i(j,\alpha)\partial^{\alpha-e_i}w_h\mathrm{d}s.
\end{split}
\end{equation}
Since $u\in H^4(\Omega)$, $w_h\in H^1(\Omega)$ and $\partial^{\alpha-e_i} w_h$, ($\alpha\in M_{i,j}$) are  tangential derivatives of the faces that orthogonal to the axe $x_i$, thus $\frac{\partial^3u}{\partial x_i^3}$ and $\partial^{\alpha-e_i} w_h$, ($\alpha\in M_{i,j}$) are continuous across faces $F'_{K,i}$, $F''_{K,i}$, we obtain
\begin{equation}
\begin{split}
%E_{x_i}(u,w_h)&=\sum_{K\in\mathcal{T}_h} I_i^K(\frac{\partial^2 u}{\partial x_i^2},\mathcal{R}^1_K\frac{\partial w_h}{\partial x_i})
\sum_{K\in\mathcal{T}_h}L^K_{i,1}&=-\sum_{K\in\mathcal{T}_h}\int_K\frac{\partial^4u}{\partial x_i^4}\sum_{\substack{1\leq j\leq d\\ j\neq i}}\sum_{\alpha\in M_{i,j}}B^K_i(j,\alpha)\partial^{\alpha-e_i}w_h\mathrm{d}x
\\
&\leq Ch^{|\alpha|-1}\sum_{K\in \mathcal{T}_h}|u|_{4,K}|\partial^{\alpha-e_i}w_h|_{0,K},
\end{split}
\end{equation}
where we have used the fact that $\max\limits_{j}\max\limits_{x\in K}B^K_i(j,\alpha)\leq Ch^{|\alpha|-1}$.

A further application of inverse estimate yields
\begin{eqnarray}\label{eq52}
%E_{x_i}(u,w_h)&\leq& Ch^2\sum_{K\in \mathcal{T}_h}|g''_{x_ix_i}|_{0,K}|\partial^{(2)}w_h|_{0,K}
\sum_{K\in\mathcal{T}_h}L^K_{i,1}&\leq& Ch^2\sum_{K\in \mathcal{T}_h}|u|_{4,K}|\nabla^2w_h|_{0,K}\nonumber
\\
&\leq&Ch^2|u|_{4,\Omega}|w_h|_{2,h}.
\end{eqnarray}
%{\color{red}
%The second term at he right-hand side of (\ref{eq41})  vanishes since $\partial^{\alpha-e_i}w_h$ is continuous between $F'_{x_i}$ and $F''_{x_i}$. Thus, the scaling argument and inverse estimate yield
%\begin{equation}
%\begin{split}
%&\sum_{K\in \mathcal{T}_h}\bigg(\int_{F''_{x_i}}-\int_{F'_{x_i}}\bigg)g\mathcal{R}^1_K\big(\frac{\partial w_h}{\partial x_i}\big)\mathrm{d}F_{x_i}\\&\leq C(\prod_{j\neq i=1}^d h_j) h_i^{-1} h_i^2h^{|\alpha|-1}(\prod_{j=1}^d h^{-1}_j) ||g''_{x_ix_i}||_{0,K}||\partial^{\alpha-e_i}w_h||_{0,K}\\&\leq Ch^2||g''_{x_ix_i}||_{0,K}||\partial^2 w_h||_{0,K}.
%\end{split}
%\end{equation}
%}
%
Then, we estimate the second term of (\ref{eq50}) $L^K_{i,2}$.
\begin{equation}
\begin{split}
L^K_{i,2}=&\int_K\frac{\partial g}{\partial x_i}\sum_{\substack{1\leq j\leq d\\ j\neq i}}\sum_{\alpha\in M_{i,j}}B^K_i(j,\alpha)(\Pi_{0,K}-Id)\partial^{\alpha}w_h\mathrm{d}x
\\
=&\int_K(Id-\Pi^K_0)\frac{\partial g}{\partial x_i}\sum_{\substack{1\leq j\leq d\\ j\neq i}}\sum_{\alpha\in M_{i,j}}B^K_i(j,\alpha)(\Pi_{0,K}-Id)\partial^{\alpha}w_h\mathrm{d}x
\\
&+\int_K\Pi_{0,K}(\frac{\partial g}{\partial x_i})\sum_{\substack{1\leq j\leq d\\ j\neq i}}\sum_{\alpha\in M_{i,j}}B^K_i(j,\alpha)(\Pi_{0,K}-Id)\partial^{\alpha}w_h\mathrm{d}x.
\end{split}
\end{equation}
According to (\ref{eq54}), since
\begin{equation*}
\int_K[(x_j-x_{j,c})^2-h_j^2](\Pi_{0,K}-Id)(c_1+c_2x_j)\;\mathrm{d}x=0,\;c_1,\;c_2\;\text{are constant coefficients},
\end{equation*}
and
\begin{equation*}
\int_K[(x_j-x_{j,c})^3-h_j^2(x_j-x_{j,c})](\Pi_{0,K}-Id)c_3\;\mathrm{d}x=0,\; c_3\text{ is a constant},
\end{equation*}
thus, we can get that
\begin{equation}
\int_K\Pi_{0,K}(\frac{\partial g}{\partial x_i})\sum_{\substack{1\leq j\leq d\\ j\neq i}}\sum_{\alpha\in M_{i,j}}B^K_i(j,\alpha)(\Pi_{0,K}-Id)\partial^{\alpha}w_h\mathrm{d}x=0.
\end{equation}
The interpolation error estimate and inverse estimate yield
\begin{equation}
\begin{split}
&\int_K(Id-\Pi_{0,K})\frac{\partial g}{\partial x_i}\sum_{\substack{1\leq j\leq d\\ j\neq i}}\sum_{\alpha\in M_{i,j}}B^K_i(j,\alpha)(\Pi_{0,K}-Id)\partial^{\alpha}w_h\mathrm{d}x\\
&\leq Ch|g|_{2,K}h^{|\alpha|-1}|\partial^{\alpha}w_h|_{0,K}
\leq Ch^{|\alpha|}|u|_{4,K}|\partial^{\alpha}w_h|_{0,K}
\\
&\leq Ch^2|u|_{4,K}|\nabla^2w_h|_{0,K}.
\end{split}
\end{equation}
Then, we can get
\begin{equation}\label{eq51}
L^K_{i,2}\leq Ch^2|u|_{4,K}|\nabla^2w_h|_{0,K}.
\end{equation}
A combination of (\ref{eq52}) and (\ref{eq51}) leads to
\begin{equation}
\begin{split}
E_{x_i}(u,w_h)=&\sum_{K\in\mathcal{T}_h}I_i^K(\frac{\partial^2 u}{\partial x_i^2},\mathcal{R}^1_K\frac{\partial w_h}{\partial x_i})
\leq Ch^2|u|_{4,\Omega}|w_h|_{2,h}.
\end{split}
\end{equation}
Similarly we obtain further,
\begin{equation}\label{eq43}
E_h(u,w_h)\leq Ch^2|u|_{4,\Omega}|w_h|_{2,h}.
\end{equation}
This, combined with the approximation error estimate, finishes the proof.
\end{proof}
\begin{rem}
Compared with \cite{LascauxLesaint} and \cite{ShiChen}, we prove the $\mathcal{O}(h^2)$ energy norm convergence rate without assuming the uniformity of the meshes. Besides, we only need the lowest regularity assumption $u\in H^4(\Omega)$.
\end{rem}

\subsection{Error analysis of the  Adini element in $L^2$ norm}
\label{sec:l2norm}

In this section, we present the lower bound estimate of the error in $L^2$ norm. This is a generalisation of the result in \cite{HuShi} to arbitrary dimension. The main result of this section is the theorem below.
\begin{thm}
\label{thm:l2err}
 Let $u$ and $u_h$ be solutions of problem (\ref{eq3}) and (\ref{eq5}), respectively. Suppose that $u \in H_0^2(\Omega) \bigcap H^s(\Omega)$, $s\geq 4$ and $s> \frac{d}{2}+1$.  Then, provided $||f||_{L^2(\Omega)}\neq 0$,
\begin{equation}
 ||u-u_h||_{L^2(\Omega)} \geq \beta h^2,
\end{equation}
where $\beta=\delta/||f||_{L^2(\Omega)}$.
\end{thm}
\begin{rem}\label{rem1}
By the embedding theorem of the Sobolev space, we need higher regularity of the solution in higher dimensions in order to guarantee $H^s(K)\subset C^1(K)$. Furthermore, it ensures the continuity of interpolation operators.
\end{rem}
\begin{rem}
For the rectangular domain $\Omega$, the condition $||f||_{L^2(\Omega)}\neq 0$ implies that $|\frac{\partial^2 u}{\partial x_i\partial x_j}|_{H^1(\Omega)}\neq 0$, $1\leq i\neq j\leq d$. In fact, if $|\frac{\partial^2 u}{\partial x_i\partial x_j}|_{H^1(\Omega)}=0$, $1\leq i\neq j\leq d$, then $u$ is of the form
\begin{equation*}
u=\sum_{i=1}^d\sum_{\substack{1\leq j\leq d\\j\neq i}}c_{ij}x_ix_j+\sum_{i=1}^d g(x_i),
\end{equation*}
for some function $g(x_i)$ with respect to $x_i$. Then the boundary condition indicates $u\equiv 0$, which contradicts $u\not\equiv 0$.
\end{rem}
We postpone the proof of Theorem \ref{thm:l2err} after several technical lemmas.

Define the global interpolation operator $\Pi_h$ to $V_h$ by
\begin{equation}
 \Pi_h|_K=\Pi_K\;\; \text{for any}\; K \in \mathcal{T}_h.
  \end{equation}
By means of Lemma \ref{lem4}, we can obtain the following crucial result.
\begin{lem} \label{lem5}
For $u \in H_0^2(\Omega) \bigcap H^s(\Omega)$, $s\geq 4$ and $s> \frac{d}{2}+1$, it holds that,
\begin{equation}\label{eq26}
  (\nabla_h^2(u-\Pi_h u),\nabla_h^2 \Pi_h u)_{L^2(\Omega)}\geq \alpha h^2,
 \end{equation}
 %where $\lim\limits_{h\rightarrow 0}\alpha_h=0.$
 for some positive constant $\alpha$, which is independent of the mesh size $h$ provided that $||f||_{L^2(\Omega)}\neq 0$ and that the mesh size is small enough.
\end{lem}
%\begin{rem}
%Herein, we need the regularity of $u$ to ensure $\Pi_h u\in C^0$ and the space of Adini element $V_h\subset C^0$.%$\Pi_h u\in V_h$, and $V_h\subset
%\end{rem}
\begin{proof}
Given any element $K$, we follow \cite{HuShi} to define $P_K v \in P_4(K)$ by
\begin{equation}\label{eq30}
 \int_K\nabla^l P_K v\,\mathrm{d}x=\int_K\nabla^l v\,\mathrm{d}x,\;l=0,1,2,3,4,
\end{equation}
 for any $v \in H^s(\Omega)$, ($s\geq 4$ and $s>\frac{d}{2}+1$). Note that the operator $P_K$ is well-defined. The interpolation operator $P_K$ has the following error estimates:
\begin{equation}\label{eq32}
\begin{split}
 |v-P_K v|_{j,K}&\leq Ch^{4-j}|v|_{4,K},\;j=0,1,2,3,4,\\
 |v-P_K v|_{j,K}&\leq Ch|v|_{j+1,K},\;j=0,1,2,3,
 \end{split}
 \end{equation}\\
provided that $v \in H^s(\Omega)$, ($s\geq 4$ and $s>\frac{d}{2}+1$). It follows from the definition of $P_K$ in (\ref{eq30}) that
\begin{equation}\label{eq31}
\nabla^4 P_K v=\Pi_{0,K}\nabla^4 v.
\end{equation}
By the aid of $P_K$, we have the following decomposition
\begin{eqnarray}
(\nabla_h^2(u-\Pi_h u),\nabla_h^2\Pi_h u)_{L^2(\Omega)}&=&\sum_{K\in\mathcal{T}_h} (\nabla_h^2(P_Ku-\Pi_K P_K u),\nabla_h^2\Pi_K u)_{L^2(K)}\nonumber\\
                                                         &&+\sum_{K\in\mathcal{T}_h} (\nabla_h^2(Id-\Pi_K)(Id-P_K)u,\nabla_h^2\Pi_K u)_{L^2(K)}\nonumber\\
                                                       &=&J_1+J_2.  \label{eq28}
\end{eqnarray}

We first analyze the first term $J_1$ on the right-hand side of (\ref{eq28}).
%Let $u=P_Kw$ and $v=\Pi_Kw$ in Lemma \ref{lem4}.
By means of Lemma \ref{lem4}, the first term $J_1$ on the right-hand side of  (\ref{eq28}) can be rewritten as
\begin{equation}
\begin{split}
J_1=&-\sum_{K\in\mathcal{T}_h}\sum_{i=1}^d\sum_{\substack{1\leq j\leq d\\j\neq i}} \frac{h_{j}^2}{3}\int_K\frac{\partial^4P_Ku}{\partial x_i^2\partial x_j^2}\frac{\partial^2\Pi_Ku}{\partial x_i^2}\,\mathrm{d}x
\\
=&-\sum_{K\in\mathcal{T}_h}\sum_{i=1}^d\sum_{\substack{1\leq j\leq d\\j\neq i}} \frac{h_{j}^2}{3}\int_K \frac{\partial^4u}{\partial x_i^2x_j^2}\frac{\partial^2u}{\partial x_i^2}\mathrm{d}x
\\
&+\sum_{K\in\mathcal{T}_h}\sum_{i=1}^d\sum_{\substack{1\leq j\leq d\\j\neq i}} \frac{h_{j}^2}{3}\int_K\frac{\partial^4(Id-P_K)u}{\partial x_i^2\partial x_j^2}\frac{\partial^2\Pi_Ku}{\partial x_i^2}\,\mathrm{d}x
\\
&+\sum_{K\in\mathcal{T}_h}\sum_{i=1}^d\sum_{\substack{1\leq j\leq d\\j\neq i}} \frac{h_{j}^2}{3}\int_K \frac{\partial^4u}{\partial x_i^2x_j^2}\frac{\partial^2(Id-\Pi_Ku)}{\partial x_i^2}\mathrm{d}x.
\end{split}
\end{equation}
Since $\frac{\partial u}{\partial x_j}\big|_{F'_{K,j}}=0$, $\frac{\partial u}{\partial x_j}\big|_{F''_{K,j}}=0$, and $\frac{\partial^3 u}{\partial^2x_i \partial x_j }\big|_{F'_{K,j}}=0$, $\frac{\partial^3 u}{\partial^2x_i \partial x_j }\big|_{F''_{K,j}}=0$, integrating by parts yields
\begin{eqnarray*}
\sum_{K\in\mathcal{T}_h}\sum_{i=1}^d\sum_{\substack{1\leq j\leq d\\j\neq i}} \frac{h_{j}^2}{3}\int_K \frac{\partial^4u}{\partial x_i^2x_j^2}\frac{\partial^2u}{\partial x_i^2}\mathrm{d}x&=
&-\sum_{K\in\mathcal{T}_h}\sum_{i=1}^d\sum_{\substack{1\leq j\leq d\\j\neq i}} \frac{h_{j}^2}{3}\int_K \big(\frac{\partial^3u}{\partial x_i^2x_j}\big)^2\mathrm{d}x
\\
&=&-\sum_{K\in\mathcal{T}_h}\sum_{i=1}^d\sum_{\substack{1\leq j\leq d\\j\neq i}} \frac{h_{j}^2}{3} \big|\big|\frac{\partial^3u}{\partial x_i^2x_j}\big|\big|^2_{L^2(K)}.
\end{eqnarray*}
By the commuting property of (\ref{eq31}),
\begin{equation*}
\frac{\partial^4(Id-P_K)u}{\partial x_i^2\partial x_j^2}=(Id-\Pi_{0,K})\frac{\partial^4u}{\partial x_i^2\partial x_j^2},\;1\leq i\neq j\leq d.
\end{equation*}
Note that
\begin{equation*}
\sum_{i=1}^d\big|\big|\frac{\partial^2\Pi_Ku}{\partial x_i^2}\big|\big|_{L^2(K)}\leq C|u|_{3,K}.
\end{equation*}
This and the error estimate of (\ref{eq32}) yield
\begin{equation}\label{eq33}
J_1=\sum_{K\in\mathcal{T}_h}\sum_{i=1}^d\sum_{\substack{1\leq j\leq d\\j\neq i}}\frac{h_{j}^2}{3} \big|\big|\frac{\partial^3u}{\partial x_i^2x_j}\big|\big|^2_{L^2(K)}+O(h^2)||(Id-\Pi_{0,K})\nabla_h^4u||^2_{L^2(K)}|u|_{3,K}.
\end{equation}
We turn to the second term $J_2$ on the right-hand side of (\ref{eq28}). By the  Poincare inequality, and the commuting property of (\ref{eq31}),
\begin{eqnarray}
|J_2|&\leq & Ch^2 \sum_{K\in\mathcal{T}_h}||\nabla_h^4(Id-P_K)u||_{L^2(K)}|u|_{3,\Omega}\nonumber\\&\leq&Ch^2||(Id-\Pi_0)\nabla_h^4 u||_{L^2(\Omega)}|u|_{3,\Omega}.\label{eq34}
\end{eqnarray}
Since the piecewise constant functions are dense in the space $L^2(\Omega)$,
\begin{equation}\label{eq35}
||(Id-\Pi_0)\nabla^4 u||_{L^2(\Omega)}\rightarrow\; 0,\;\;when\;\;h\rightarrow \;0.
\end{equation}
Summation of (\ref{eq33}), (\ref{eq34}) and (\ref{eq35}) completes the proof.
\end{proof}

Again, the lemma below can be found in \cite{HuShi}.
\begin{lem}
Let $u$ and $u_h$ be solutions of problem (\ref{eq3}) and (\ref{eq5}), respectively. Then,
  \begin{equation}\label{eq19}
  \begin{split}
  (-f,u-u_h)_{L^2(\Omega)}&=a_h(u,\Pi_h u-u_h)-(f,\Pi_h u-u_h)_{L^2(\Omega)}\\&\;\;+a_h(u-\Pi_h u, u-\Pi_h u)+a_h(u-\Pi_h u, u_h-\Pi_h u)\\&\;\;+2(f,\Pi_h u-u)_{L^2(\Omega)}+2a_h(u-\Pi_h u, \Pi_h u).
  \end{split}
  \end{equation}
\end{lem}

\paragraph{\textbf{Proof of Theorem \ref{thm:l2err}}}

  It follows from (\ref{eq43}) that
  \begin{equation}\label{eq23}
  a_h(u,\Pi_h u-u_h)-(f,\Pi_h u-u_h)_{L^2(\Omega)}\leq Ch^2|u|_{4,\Omega}|\Pi_h u-u_h|_h\leq Ch^4|u|^2_{4,\Omega}.
  \end{equation}
  By the Cauchy-Schwarz inequality and the error estimate (\ref{eq18}), it yields
\begin{equation}
a_h(u-\Pi_h u, u-\Pi_h u)+2(f,\Pi_h u-u)_{L^2(\Omega)}\leq Ch^4(|u|_{4,\Omega}+||f||_{L^2(\Omega)})|u|_{4,\Omega},\label{eq24}
\end{equation}
\begin{equation}
a_h(u-\Pi_h u, u_h-\Pi_h u)\leq Ch^4|u|^2_{4,\Omega}.\label{eq25}
\end{equation}
 % In order to obtain the error estimates of other terms of (\ref{eq19}), {\color{red} here estimate about $a_h(u-\Pi_hu,\Pi_hu)$.}
 The error estimate of the last term of (\ref{eq19}) by Lemma \ref{lem5} gives
 \begin{equation}
 \alpha h^2\leq a_h(u-\Pi_hu,\Pi_hu).\label{eq100}
 \end{equation}
Hence, a combination of (\ref{eq19})-(\ref{eq100}) leads to
\begin{equation}
 (-f, u-u_h)_{L^2(\Omega)}\geq \delta h^2.
\end{equation}
for some positive constant $\delta$, which is independent of the mesh size $h$ and the mesh size is small enough.

Therefore,
\begin{eqnarray*}
||u-u_h||_{L^2(\Omega)}&=&\sup_{0\neq w\in L^2(\Omega)}\frac{(w,u-u_h)_{L^2(\Omega)}}{||w||_{L^2(\Omega)}}\\&\geq&\frac{(-f,u-u_h)_{L^2(\Omega)}}{||-f||_{L^2(\Omega)}}\geq \delta/||f||_{L^2(\Omega)}h^2.
\end{eqnarray*}
This finishes the proof.
\qed

\section{Numerical examples}
\label{sec:numerical}

In this section, we present some numerical results of the three-dimensional Adini element by congruence partition of cubic meshes and non-congruence partition of domain $\Omega$ to demonstrate our theoretical results. Herein, we give $u_1(x,y,z)=sin^2(\pi x)sin^2(\pi y)sin^2(\pi z)$ and $u_2(x,y,z)=x^2(1-x)^2y^2(1-y)^2z^2(1-z)^2$ as the exact solution of problem (\ref{eq3}), respectively. One can see the errors and the rate of convergence computed by uniform cubic meshes with the meshsize $h=\frac{1}{N}$ for some integer $N$ in Figure 2. One can also see the errors and the rate of convergence computed by non-congruence meshes in Figure 3 for logarithmic plot.

\begin{figure}[!hbt]
  \centering
  \includegraphics[width=5in, height=3in]{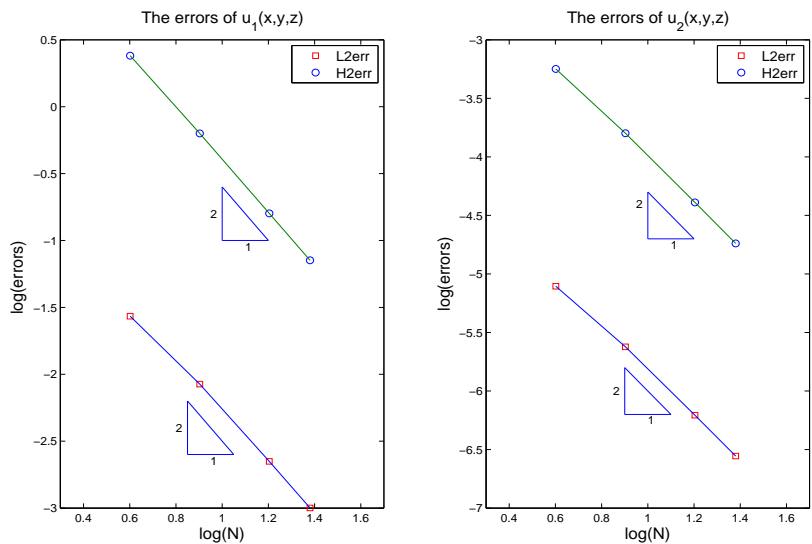}
  \caption{The errors in $L^2$ and $H^2$ norms for $u_1(x,y,z)$ and $u_2(x,y,z)$ by uniform cubic meshes}

\end{figure}

\begin{figure}[!hbt]
  \centering
  \includegraphics[width=5in, height=3in]{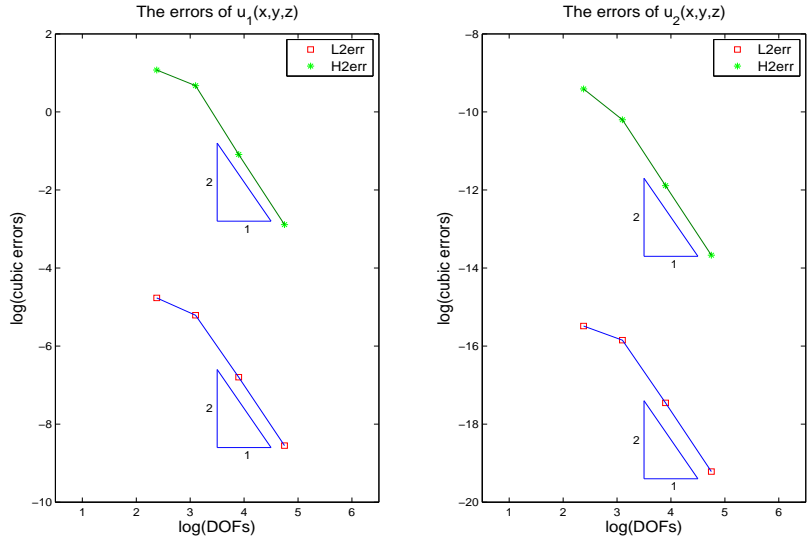}
  \caption{The errors in $L^2$ and $H^2$ norms for $u_1(x,y,z)$ and $u_2(x,y,z)$ by non-congruence meshes}

\end{figure}

\section{Concluding remarks}
\label{sec:conclusion}

In this paper, we studied the accuracy of the Adini element as a discretization scheme for fourth order problem in any dimension. We showed that the convergence rate is of $\mathcal{O}(h^2)$ order in energy norm in any dimension, and moreover, we show that the convergence rate can not be non-trivially higher than $\mathcal{O}(h^2)$ order in integral norm in any dimension. By the Poincare inequality, we arrive at the conclusion that the convergence rate of Adini element for discretising fourth order problem is $\mathcal{O}(h^2)$ in $L^2$, $H^1$ and the energy norm. This presents a complete exploration of the capacity of the scheme.

The results provided in this paper are optimal in two-folded. On one hand, the full convergence rate of the energy norm is established under the assumption $u\in H^4(\Omega)$, which is standard and of the lowest regularity. On the other hand, combining the two results together illustrates that neither of these two can be improved. The analysis of this paper has been sharp and economic enough.

It is somehow surprising to observe the convergence rates in different norms are of the same order. This is because the Adini element function is not moment continuous across the edges, but it possesses internal symmetry on every cell when the grid is of tensor type. This point of view will hint us on other elements.

\end{document}